\documentclass[a4paper]{amsart}
\usepackage[T1]{fontenc} 
\usepackage{amsmath,amsthm}
\usepackage[dvips]{graphicx}
\usepackage{todonotes}
\usepackage{color}
\usepackage{tikz}
\usetikzlibrary{arrows,positioning,shapes,matrix,backgrounds}
\usetikzlibrary{fit,calc,decorations.pathreplacing}

\usepgflibrary{arrows}
\usepackage{soul}
\usepackage{amsfonts,amssymb}
\usepackage{geometry}
\usepackage{hyperref}
\usepackage{stmaryrd}
\usepackage{fancyhdr}
\usepackage{url}
\usepackage{todonotes}
\usepackage{dsfont}
\usepackage[utf8]{inputenc} 
\usepackage{enumerate}

\theoremstyle{definition}

\newtheorem*{thm*}{Theorem}

\newtheorem{prop}{Proposition}[section]

\newtheorem{lemma}[prop]{Lemma}

\newtheorem{thm}[prop]{Theorem}

\newtheorem{corollary}[prop]{Corollary}

\newtheoremstyle{pourlesremarques}{\topsep}{\topsep}{\normalfont}{}{\bfseries}{.}{ }{}
\theoremstyle{pourlesremarques}

\newtheorem*{rem*}{Remark}
\newtheoremstyle{pourlesexemples}{\topsep}{\topsep}{\normalfont}{}{\bfseries}{.}{ }{}
\theoremstyle{pourlesexemples}

\def\presuper#1#2%
  {\mathop{}%
   \mathopen{\vphantom{#2}}^{#1}%
   \kern-\scriptspace%
   #2}

\newcommand{\li}{\mathfrak{l}}
\newcommand{\St}{{\operatorname{St}}}

\newcommand{\cusp}{{\operatorname{cusp}}}

\def\Rep{\operatorname{Rep}}
\def\Irr{\operatorname{Irr}}
\newcommand{\Indec}{\operatorname{Indec}}
\newcommand{\Nilp}{\operatorname{Nilp}}

\newcommand{\Image}{\mathrm{Im}}

\newcommand{\Hom}{\operatorname{Hom}}

\newcommand{\WD}{\mathrm{D}}
\newcommand{\W}{\mathrm{W}}
\newcommand{\G}{\mathcal{G}}

\newcommand{\GL}{\operatorname{GL}}

\newcommand{\SL}{\operatorname{SL}}

\newcommand{\Ker}{\mathrm{Ker}}

\newcommand{\La}{\mathrm{L}}
\newcommand{\V}{\mathrm{V}}
\newcommand{\C}{\mathrm{C}}
\newcommand{\CV}{\mathrm{CV}}

\newcommand{\N}{\mathbb{N}}

\newcommand{\Z}{\mathbb{Z}}

\newcommand{\Fl}{\overline{\mathbb{F}_\ell}}

\def\ss{\operatorname{ss}}

\def\GL{\operatorname{GL}}

\def\\Hom{\operatorname{\Hom}}
\def\Irr{\operatorname{Irr}}
\def\Cyc{\mathcal{C}}

\def\St{\operatorname{St}}

\def\Rep{\operatorname{Rep}}

\def\leq{\leqslant}
\def\geq{\geqslant}

\def\Rep{\operatorname{Rep}}

\def\presuper#1#2%
  {\mathop{}%
   \mathopen{\vphantom{#2}}^{#1}%
   \kern-\scriptspace%
   #2}
\makeatletter
\DeclareRobustCommand{\rvdots}{%
  \vbox{
    \baselineskip4\p@\lineskiplimit\z@
    \kern-\p@
    \hbox{.}\hbox{.}\hbox{.}  \hbox{.}
  }}
\makeatother

\title{A characterization of the relation between two $\ell$-modular correspondences}
\author{Robert Kurinczuk and Nadir Matringe}

\date{\today}
\begin{document}

\maketitle

\begin{abstract}
Let $F$ be a non archimedean local field of residual characteristic $p$ and $\ell$ a prime number different from $p$. Let 
$\V$ denote Vign\'eras'~$\ell$-modular local Langlands correspondence \cite{Viginv}, between irreducible~$\ell$-modular representations of $\GL_n(F)$ and $n$-dimensional $\ell$-modular Deligne representations of 
the Weil group~$\W_F$.  In \cite{KM18}, enlarging the space of parameters to Deligne representations with non necessarily nilpotent operators, we proposed a modification of the correspondence of Vign\'eras into a 
correspondence $\C$ compatible with the formation of local constants in the generic case. In this note, following a remark of Alberto M\'inguez,
we characterize the modification $\C\circ \V^{-1}$ by a short list of natural properties.
\end{abstract}

\section{Introduction}

Let~$F$ be a non-archimedean local field with finite residue field of cardinality~$q$, a power of a prime~$p$, and 
$\W_F$ the Weil group of~$F$. Let $\ell$ be a prime number different from $p$. The $\ell$-modular local Langlands correspondence established
by Vignéras in \cite{Viginv} is a bijection from isomorphism classes of smooth irreducible representations of $\GL_n(F)$ and $n$-dimensional Deligne representations 
(Section \ref{section Deligne}) of the Weil group $\W_F$ with nilpotent monodromy operator. It is uniquely characterized by a non-naive compatibility 
with the $\ell$-adic local Langlands correspondence (\cite{LRS}, \cite{HT01}, \cite{HenniartLLC}, \cite{Sch}) under reduction modulo $\ell$, involving twists by Zelevinsky involutions. In \cite{KM18}, at the cost of 
having a less direct compatibility with reduction modulo $\ell$, we proposed a modification of the correspondence $\V$ of Vignéras, 
by in particular enlarging its target to the larger space of Deligne representations with non necessarily nilpotent monodromy operator 
(it is a particularity of the $\ell$-modular setting that such operators can live outside the nilpotent world). The modified correspondence 
$\C$ is built to be compatible with local constants on both sides of the corrspondence (\cite{KM17} and \cite{KM18}) and we proved that it is indeed the case 
for generic representations in \cite{KM18}. Here, we show in Section \ref{section characterization} that if we expect a correspondence to have such a property, and some other natural 
properties, then it will be uniquely determined by $\V$. Namely we characterize the map $\C\circ \V^{-1}$ by a list of five 
properties in Theorem \ref{theorem characterization}. The map $\C\circ \V^{-1}$ endows the image of $\mathrm{C}$ with a semiring structure because the image of $\V$ is 
naturally equipped with semiring laws. We end this note by 
studying this structure from a different point of view in Section \ref{section semiring}.

\section{Preliminaries}

Let~$\nu:\W_F\rightarrow \Fl^\times$ be the unique character trivial on the inertia subgroup of~$\W_F$ and sending a geometric Frobenius element to~$q^{-1}$, it corresponds to the normalized 
absolute value $\nu:F^\times\rightarrow \Fl^\times$ via local class field theory. 

We consider only smooth representations of locally compact groups, which unless otherwise stated will be considered on $\Fl$-vector spaces.  For~$\G$ a locally compact topological group, we 
let~$\Irr(\G)$ denote the set of isomorphism classes of irreducible representations of $\G$.

\subsection{Deligne representations}\label{section Deligne}
We follow \cite[Section 4]{KM18}, but slightly simplify some notation.  A \textit{ Deligne-representation} of~$\W_F$ is a pair~$(\Phi,U)$ where~$\Phi$ is a finite dimensional 
semisimple~representation of~$\W_F$, and~$U\in\Hom_{\W_F}(\nu\Phi,\Phi)$; we call~$(\Phi,U)$ 
 \emph{nilpotent} if~$U$ is a nilpotent endomorphism over~$\Fl$.  
 
 The set of morphisms between Deligne representations~$(\Phi,U),(\Phi',U')$ (of~$\W_F$) is given by~$\Hom_{\WD}(\Phi,\Phi')=\{f\in\Hom_{\W_F}(\Phi,\Phi'): f\circ U=U'\circ f\}$. 
 This leads to notions of irreducible and indecomposable Deligne representations.  We refer to \cite[Section 4]{KM18}, for the (standard) definitions of dual and direct sums of Deligne representations.

We let~$\Rep_{\WD,\ss}(\W_F)$ denote the set of isomorphism classes of~Deligne-representations; and $\Indec_{\WD,\ss}(\W_F)$ (resp.~$\Irr_{\WD,\ss}(\W_F)$,~$\Nilp_{\WD,\ss}(\W_F)$) denote
the set of isomorphism classes of indecomposable (resp. irreducible, nilpotent)~Deligne representations.  Thus
\[\Irr_{\WD,\ss}(\W_F)\subset \Indec_{\WD,\ss}(\W_F)\subset \Rep_{\WD,\ss}(\W_F),\qquad \Nilp_{\WD,\ss}(\W_F)\subset \Rep_{\WD,\ss}(\W_F).\]
Let~$\Rep_{\ss}(\W_F)$ denote the set of isomorphism classes of semisimple representations of~$\W_F$, we have a canonical 
map~$\mathrm{Supp}_{\W_F}:\Rep_{\WD,\ss}(\W_F)\rightarrow \Rep_{\ss}(\W_F),\quad~(\Phi,U)\mapsto \Phi$; we call~$\Phi$ the~\emph{$\W_F$-support} of~$(\Phi,U)$.

For $\Psi\in \Irr(\W_F)$ we denote by~$o(\Psi)$ the cardinality of the \emph{irreducible line}~$\Z_{\Psi}=\{\nu^k\Psi, k\in \Z\}$; it divides 
the order of $q$ in $\mathbb{F}_\ell^\times$ hence is prime to~$\ell$. We let 
$\li(\W_F)=\{\Z_\Psi:\Psi\in\Irr(\W_F)\}$.

The fundamental examples of non-nilpotent Deligne representation are the \emph{cycle representations}: let~$I$ be an isomorphism from~$\nu^{o(\Psi)}\Psi$ to~$\Psi$ 
and define~$\Cyc(\Psi,I)=(\Phi(\Psi),C_I)\in \Rep_{\ss}(\WD,\Fl)$ by
\[\Phi(\Psi)=\bigoplus_{k=0}^{o(\Psi)-1} \nu^{k}\Psi,\quad C_I(x_0,\dots,x_{o(\Psi)-1})=( I(x_{o(\Psi)-1}),x_0,\dots,x_{o(\Psi)-2}),~x_k\in\nu^k\Psi.\]
Then $\Cyc(\Psi,I)\in \Irr_{\ss}(\WD, \Fl)$ and its isomorphism class only depends on~$(\Z_\Psi,I)$, by \cite[Proposition 4.18]{KM18}.  

To remove dependence on~$I$, in \cite[Definition 4.6 and Remark 4.9]{KM18} we define an equivalence relation $\sim$ on $\Rep_{\WD,\ss}(\W_F)$.  The equivalence class of~$\Cyc(\Psi,I)$ is independent of $I$, and we set
\[\Cyc(\Z_\Psi):=[\Cyc(\Psi,I)]\in [\Irr_{\WD,\ss}(\W_F)].\]

The sets $\Rep_{\WD,\ss}(\W_F)$, $\Irr_{\WD,\ss}(\W_F)$, $\Indec_{\WD, \ss}(\W_F)$, and $\Nilp_{\WD, \ss}(\W_F)$ are unions of $\sim$-classes, and if~$X$ denotes 
any of them we set~$[X]:=X/\sim$.  Similarly, for~$(\Phi,U)\in\Rep_{\WD,\ss}(\W_F)$ we write~$[\Phi,U]$ for its equivalence class in~$[\Rep_{\WD,\ss}(\W_F)]$. 
On~$\Nilp_{\WD, \ss}(\W_F)$ the equivalence relation~$\sim$ coincides with equality.

The operations $\oplus$ and~$(\Phi,U)\mapsto (\Phi,U)^\vee$ on $\Rep_{\WD,\ss}(\W_F)$ descend to $[\Rep_{\WD,\ss}(\W_F)]$.  Tensor products are more subtle; 
for example, tensor products of semisimple representations of~$\W_F$ are not necessarily semisimple.  We define a semisimple tensor 
product operation~$\otimes_{\ss}$ on $[\Rep_{\WD,\ss}(\W_F)]$ in \cite[Section 4.4]{KM18}, turning $([\Rep_{\WD,\ss}(\W_F)],\oplus,\otimes_{\ss}$) into an abelian semiring.

The basic non-irreducible examples of elements of~$\Nilp_{\WD,\ss}(\W_F)$ are called \emph{segments}:  For $r\geq 1$, set~$[0,r-1]:=(\Phi(r),N(r))$, where
\begin{align*} 
\Phi(r)=\bigoplus_{k=0}^{r-1} \nu^k,\quad N(r)(x_0,\ldots,x_{r-1})&=(0,x_0,\ldots,x_{r-2}),~x_k\in\nu^k.\end{align*}

We now recall the classification of equivalence classes of Deligne representations of~$\W_F$ of \cite{KM18}. 

\begin{thm}[{\cite[Section 4]{KM18}}]\label{theorem classification}
\begin{enumerate} 
\item Let~$\Phi\in \Irr_{\WD,\ss}(\W_F)$, then there is either a unique $\Psi \in  \Irr(\W_F)$ such that $\Phi=\Psi$, or a unique 
irreducible line $\Z_\Psi$ such that $[\Phi]=\Cyc(\Z_\Psi)$.
\item Let~$[\Phi,U]\in[\Indec_{\WD,\ss}(\W_F)]$, then there exist a unique $r\geq 1$ and a unique $\Theta\in [\Irr_{\WD,\ss}(\W_F)]$ such that~$[\Phi,U]=[0,r-1]\otimes_{\ss} \Theta$.
\item Let~$[\Phi,U]\in[\Rep_{\WD,\ss}(\W_F)]$, there there exist~$[\Phi_i,U_i]\in[\Indec_{\WD,\ss}(\W_F)]$ for~$1\leqslant i\leqslant r$ such that~$[\Phi,U]=\bigoplus_{i=1}^r[\Phi_i,U_i]$.
\end{enumerate}
\end{thm}

We recall the following classical result about tensor products of segments.

\begin{lemma}\label{lemma tensor product decomposition}
For $n\geq m\geq 1$], one has 
\[[0,n-1]\otimes_{\ss} [0,m-1]=[0,n+m-2]\oplus [1,n+m-3] \oplus \dots \oplus [m-1,n-1].\] 
\end{lemma}
\begin{proof}
In the $\ell$-adic setting the corresponding statement can be translated into a statement on tensor product of irreducible representations of $\SL_2(\mathbb{C})$, which is well-known 
and easily checked by the highest weight theory. The 
corresponding $\ell$-adic representations are integral and~$\otimes_{\ss}$ coincides with the standard tensor product on segments, 
the statement follows from reduction modulo $\ell$ of the~$\ell$-adic result.\end{proof}

\subsection{$\La$-factors}

We set $\Irr_{\cusp}(\GL(F)):=\coprod_{n\geq 0} \Irr_{\cusp}(\GL_n(F))$ where $\Irr_{\cusp}(\GL_n(F))$ is the set of isomorphism classes of irreducible cuspidal representations of $\GL_n(F)$.

Let $\pi$ and $\pi'$ be a pair of cuspidal representations of $\GL_n(F)$ and $\GL_m(F)$ respectively. We denote by $\La(X,\pi,\pi')$ the Euler factor attached to this pair 
in \cite{KM17} via the Rankin-Selberg method. 
We recall that a cuspidal representation of $\GL_n(F)$ is called \textit{banal} if $\nu\otimes \pi\not\simeq \pi$. The following is a part of~\cite[Theorem 4.9]{KM17}.

\begin{prop}\label{theorem L functions}
Let $\pi,\pi'\in\Irr_{\cusp}(\GL(F))$. If $\pi$ or $\pi'$ is non-banal, then $\La(X,\pi,\pi')=1$. 
\end{prop}

%
Let $[\Phi,U]\in [\Rep_{\WD,\ss}(\W_F)]$, for brevity from now on we often denote such a class just by~$\Phi$, we denote by 
$\La(X,\Phi)$ the $\La$-factor attached to it in 
\cite[Section 5]{KM18}, their most basic property is that \[\La(X,\Phi\oplus \Phi')= \La(X,\Phi)\La(X,\Phi')\] for 
$\Phi$ and $\Phi'$ in $[\Rep_{\WD,\ss}(\W_F)]$. We need the following property of such factors.

\begin{lemma}\label{lemma pole at zero}
Let $\Psi\in \Irr(\W_F)$ and $a\leq b$ be integers, put~$\Phi=[a,b]\otimes_{\ss} \Psi$ and~$\Phi'=[-b,-a]\otimes_{\ss}\Psi^\vee$, then $\La(X,\Phi\otimes_{\ss}\Phi')$ has a pole at $X=0$.
\end{lemma}
\begin{proof}
According to \cite[Lemma 5.7]{KM18}, it is sufficient to prove that $\La(X,\Psi\otimes_{\ss}\Psi^\vee)$ has a pole at $X=0$ for $\Psi\in \Irr(\W_F)$, 
but this property follows from the definition of the $\La$-factor in question, and the fact that $\Psi\otimes_{\ss} \Psi^\vee$ contains a nonzero vector fixed by $\W_F$.
\end{proof}

\subsection{The map $\CV$}\label{subsection CV}
For $\Psi\in \Irr(\W_F)$ we set~$\St_0(\Z_\Psi)=\bigoplus_{k=0}^{o(\Psi)-1} \nu^k\Psi$.  By Theorem \ref{theorem classification}, an element $\Phi\in \Nilp_{\WD,\ss}(\W_F)$ has a unique decomposition
 \[\Phi=\Phi_{\text{acyc}}\oplus\bigoplus_{k\geq 1, \Z_\Psi\in \li(\W_F)} [0,k-1]\otimes_{\ss} n_{\Z_\Psi,k} \St_0(\Z_{\Psi}),\] 
where for all $k\geq1$ and~$\Z_\Psi\in\li(\W_F)$, $\Phi_{\text{acyc}}$ has no summand isomorphic to~$[0,k-1]\otimes_{\ss} \St_0(\Z_{\Psi})$; 
i.e. we have separated~$\Phi$ into an acyclic and a cyclic part.  Then following \cite[Section 6.3]{KM18}, we set:
\[\CV(\Phi)=\Phi_{\text{acyc}}\oplus\bigoplus_{k\geq 1, \Z_\Psi\in \li(\W_F)} [0,k-1]\otimes_{\ss} 
 n_{\Z_\Psi,k} \Cyc(\Z_{\Psi}).\] 
We denote by $\C_{\WD,\ss}(\W_F)$ the image of $\CV:\Nilp_{\WD,\ss}(\W_F)\rightarrow [\Rep_{\WD,\ss}(\W_F)]$, and call $\C_{\WD,\ss}(\W_F)$ the set of \textit{$\C$-parameters}. 

\subsection{$\ell$-modular local Langlands}
We let $\Irr(\GL(F))=\coprod_{n\geq 0} \Irr(\GL_n(F))$ where $\Irr(\GL_n(F))$ denotes the set of 
isomorphism classes of irreducible representations of $\GL_n(F)$. 

In \cite{Viginv}, Vign\'eras introduces the~$\ell$-modular local Langlands correspondence: a bijection 
\[\V:\Irr(\GL(F)) \rightarrow \Nilp_{\WD,\ss},\] 
characterized in a non-naive way by reduction modulo~$\ell$. For this note, we recall~$\mathrm{Supp}_{\W_F}\circ\V$, the \emph{semisimple $\ell$-modular local Langlands correspondence} of 
Vign\'eras, induces a bijection between supercuspidal supports elements of~$\Irr(\GL_n(F))$ and~$\Rep_{\ss}(\W_F)$ compatible with reduction modulo~$\ell$.

In \cite{KM18}, we introduced the bijection  
\[\C=\CV\circ \V:\Irr(\GL(F))\rightarrow \C_{\WD,\ss}(\W_F);\]
which satisfies~$\mathrm{Supp}_{\W_F}\circ\V=\mathrm{Supp}_{\W_F}\circ\C$.  Moreover, the correspondence~$\C$ is compatible with the formation of~$\La$-factors for generic representations, 
a property~$\V$ does not share; in the cuspidal case:

\begin{prop}[{\cite[Proposition 6.13]{KM18})}]
For $\pi$  and $\pi'$ in $\Irr_{\cusp}(\GL(F))$ one has $\La(X,\pi,\pi')=\La(X,\C(\pi),\C(\pi'))$.   
\end{prop}

We note another characterization of non-banal cuspidal representations: 

\begin{prop}[{\cite[Sections 3.2 and 6.2]{KM18}}]\label{proposition non banal cusp}
A representation $\pi\in \Irr_{\cusp}(\GL(F))$ is non-banal if and only if $\V(\pi)=\ell^k\St_0(\Z_\Psi)$, or equivalently $\C(\pi)=\ell^k\Cyc(\Z_\Psi)$, for some $k\geq 0$ and $\Psi\in \Irr(\W_F)$.
\end{prop}

Amongst non-banal cuspidal representations, those for which $k=0$ in the above statement, shall play a special role in our characterization.  We denote by $\Irr_{\cusp}^\star(\GL(F))$ the subset of $\Irr_{\cusp}(\GL(F))$ consisting of those~$\pi\in\Irr_{\cusp}(\GL(F))$ such that~$\C(\pi)=\Cyc(\Z_\Psi)$, for some~$\Psi\in \Irr(\W_F)$

\section{The characterization}\label{section characterization}
In this section, we provide a list of natural properties which characterize~$\CV:\Nilp_{\WD,\ss}(\W_F)\rightarrow [\Rep_{\WD,\ss}(\W_F)]$. 

\begin{prop}\label{prop observation 1}
Let $\CV':\Nilp_{\WD,\ss}(\W_F)\rightarrow [\Rep_{\WD,\ss}(\W_F)]$ be any map, and~$\C':=\CV'\circ \V$.  Suppose
\begin{enumerate}[(1)]
\item \label{a} $ \mathrm{Supp}_{\W_F}\circ \C'$ is the semisimple $\ell$-modular local Langlands correspondence of Vign\'eras; in other words,~$\CV'$~preserves the $\W_F$-support;
\item \label{b} $\C'$ (or equivalently~$\CV'$) commutes with taking duals;
\item \label{c}  $\La(X,\pi,\pi^\vee)=\La(X,\C'(\pi),\C'(\pi)^\vee)$ for all non-banal representations~$\pi \in \Irr_{\cusp}^*(\GL(F))$. \end{enumerate}
Then for all $\Psi\in \Irr(\W_F)$, one has $\CV'(\St_0(\Z_\Psi))=\Cyc(\Z_\Psi)$.
\end{prop}
\begin{proof}
Thanks to (\ref{a}),~$\CV'(\St_0(\Z_\Psi))$ has $\W_F$-support 
$\bigoplus_{k=0}^{o(\Psi)-1}\nu^k\Psi$. Hence, by Theorem \ref{theorem classification}, its image under $\CV'$ is either $\Cyc(\Z_\Psi)$ or a sum of Deligne representations of the 
form $[a,b]\otimes_{\ss} \Psi$ for $0\leq a \leq b \leq o(\Psi)-1$. If we are in the second situation, writing~$\CV'(\St_0(\Z_\Psi))=([a,b]\otimes_{\ss} \Psi)\oplus W$, we 
have $\CV'(\St_0(\Z_\Psi))^\vee=([-b,-a]\otimes_{\ss} \Psi^\vee)\oplus W^\vee$, thanks to~(\ref{b}). However, writing $\tau$ for the non-banal cuspidal representation $\V^{-1}(\St_0(\Z_\Psi))$, we have~$\La(X,\tau,\tau^\vee)=1$ according to Theorem \ref{theorem L functions} and Proposition \ref{proposition non banal cusp}, whereas 
\begin{align*}\La(X,\C(\tau),\C(\tau^\vee))&=\La(X,(([a,b]\otimes_{\ss} \Psi)\oplus W) \otimes_{\ss} (([-b,-a]\otimes_{\ss}\Psi^\vee)\oplus W^\vee))\\
&= \La(X,([a,b]\otimes_{\ss} \Psi)\otimes_{\ss} ([-b,-a]\otimes_{\ss}\Psi^\vee)) \La'(X)\end{align*} for $\La'(X)$ an Euler factor. Now, observe that 
$\La(X,([a,b]\otimes_{\ss} \Psi) \otimes_{\ss} ([-b,-a]\otimes_{\ss}\Psi^\vee))$ 
has a pole at $X=0$ according to Lemma \ref{lemma pole at zero}, hence cannot be equal to $1$. The conclusion of this discussion, according to (\ref{c}) 
is~$\CV'(\St_0(\Z_\Psi))=\Cyc(\Z_\Psi)$.
\end{proof}

It follows that (\ref{a})-(\ref{c}) characterize~$\C\vert_{\Irr_{\cusp}^*(\GL(F))}$ without reference to Vign\'eras' correspondence~$\V$.  

On the other hand any map $\CV'$ satisfying~(\ref{a})-(\ref{c}) must send each $\nu^k\Psi$ to itself if $o(\Psi)>1$ by~(\ref{a}).  So there is no 
chance that $\CV'$ will preserve direct sums because~$\CV'(\bigoplus_{k=0}^{o(\Psi)-1}\nu^k\Psi)\neq \Cyc(\Z_\Psi)$.  In particular any compatibility property 
of~$\CV'$ with direct sums will have to be non-naive.  Here is our characterization of the map $\CV$:

\begin{thm}\label{theorem characterization}
Suppose $\CV':\Nilp_{\WD,\ss}(\W_F)\rightarrow [\Rep_{\WD,\ss}(\W_F)]$ satisfies (\ref{a})-(\ref{c}) of Proposition \ref{prop observation 1}, and suppose moreover
\begin{enumerate}[(A)]
\item \label{d} If~$\Phi'\in\Image(\CV')$ and~$\Phi'=\Phi_1'\oplus\Phi_2'$ in~$[\Rep_{\WD,\ss}(\W_F)]$ then~$\Phi'_1,\Phi'_2\in\Image(\CV')$.
 Moreover, if~$\Phi'=\CV'(\Phi)$,~$\Phi'_i=\CV'(\Phi_i)$ for~$\Phi, \Phi_i\in \Nilp_{\WD,\ss}(\W_F)$, and~$\Phi'=\Phi'_1\oplus\Phi'_2$, then~$\Phi=\Phi_1\oplus \Phi_2$.
\item \label{e} $\CV'([0,j-1]\otimes_{\ss} \Phi)=[0,j-1]\otimes_{\ss}\CV'(\Phi)$ for $j\in \N_{\geq 1}$ and $\Phi\in \Nilp_{\WD,\ss}(\W_F)$.
\end{enumerate}
Then $\CV'=\CV$.
\end{thm}
\begin{proof}
For $\Psi\in \Irr(\W_F)$, it follows at once from Proposition \ref{prop observation 1} and (\ref{e}) that 
\begin{align*}
\CV'([0,j-1]\otimes_{\ss} \Psi)&=[0,j-1]\otimes_{\ss} \Psi,\quad\text{ if $o(\Psi)>1$ and}\\
\CV'([0,j-1]\otimes_{\ss} \St_0(\Z_\Psi))&=[0,j-1]\otimes_{\ss} \Cyc(\Z_\Psi).\end{align*}

Next we prove that $\Image(\CV')\subset \C_{\WD,\ss}(\W_F)$. By (\ref{d}), an element of $\Image(\CV')$ can be decomposed as a direct 
sum of elements in $\Image(\CV')\cap [\Indec_{\WD,\ss}]$, and (\ref{d}) reduces the proof of the 
inclusion $\Image(\CV')\subset \C_{\WD,\ss}(\W_F)$ to showing that $[0,j-1]\otimes_{\ss}\St_0(\Psi)\notin \Image(\CV')$ for $\Psi\in \Irr(\W_F)$, $j\geq 1$.

We first assume that $o(\Psi)=1$, so $\St_0(\Psi)=\Psi$. The only possible pre-image of $\Psi$ by $\CV'$ is $\Psi$ by (\ref{a}), however $\CV'(\Psi)=\Cyc(\Z_\Psi)$ by 
Proposition \ref{prop observation 1} so $\St_0(\Psi)\notin \Image(\CV')$.  Now suppose $[0,j-1]\otimes_{\ss} \Psi \in \Image(\CV')$ for $j\geq 2$, then by (\ref{e}) this would imply that 
$[0,j-1]\otimes_{\ss} [0,j-1]\otimes_{\ss} \Psi \in \Image(\CV')$, hence that 
\[[0,j-1]\otimes_{\ss} [0,j-1]\otimes_{\ss} \Psi= [0,2j-2]\otimes_{\ss} \Psi\oplus\dots \oplus [j-1,j-1]\otimes_{\ss} \Psi \]
also belongs to $\Image(\CV')$ thanks to Lemma \ref{lemma tensor product decomposition}. However as $o(\Psi)=1$, the Deligne 
representation $[j-1,j-1]\otimes_{\ss} \Psi$ is nothing else than $\Psi$, which does not belong to $\Image(\CV')$, contradicting (\ref{d}).

 If 
$o(\Psi)>1$, then $\CV'(\nu^k\Psi)=\nu^k\Psi$. If $\St_0(\Psi)$ belonged to $\Image(\CV')$ then (\ref{d}) would imply 
that~$\St_0(\Psi)=\CV'(\bigoplus_{k=0}^{o(\Psi)-1} \nu^k \Psi)$, which is not the case thanks to Proposition \ref{prop observation 1}. To see 
that $[0,j-1]\otimes_{\ss}\St_0(\Psi)\notin \Image(\CV')$ for all $j\geq 2$ we use the same trick as in the 
$o(\Psi)=1$ case.

Now take $\Phi\in \Nilp_{\WD,\ss}$, as we just noticed $\CV'(\Phi)$ is a $\C$-parameter and we write it 
\begin{align*}
\CV'(\Phi)&= \CV'(\Phi)_{\text{acyc}}\oplus\bigoplus_{k\geq 1, \Z_\Psi\in \li(\W_F)} [0,k-1]\otimes_{\ss} n_{\Z_\Psi,k} \Cyc(\Z_{\Psi})
 \end{align*}
as in Section \ref{subsection CV}, where for each irreducible line~$\mathbb{Z}_\Psi$ we have fixed an irreducible~$\Psi\in\mathbb{Z}_{\Psi}$.  
Then (\ref{d}) and the beginning of the proof imply that 
\[
\Phi=\CV'(\Phi)_{\text{acyc}}\bigoplus_{k\geq 1, \Z_\Psi\in \li(\W_F)} [0,k-1]\otimes_{\ss} 
 n_{\Z_\Psi,k} \St_0(\Z_{\Psi}) ,\] hence that 
$\CV'(\Phi)=\CV(\Phi)$.
\end{proof}

\section{The semiring structure on the space of $\C$-parameters}\label{section semiring}
As~$(\Nilp_{\WD,\ss}(\W_F),\oplus,\otimes_{\ss})$ is a semiring, the map $\CV$ endows $\C_{\WD,\ss}(\W_F)$ with a semiring structure by transport of structure. We show that this semiring structure on $\C_{\WD,\ss}(\W_F)$ can be obtained without referring to $\CV$ directly, thus shedding a slightly different light on the map $\CV$.

We denote by $\G(\Rep_{\WD,\ss}(\W_F))$ the Grothendieck group of the monoid $([\Rep_{\WD,\ss}(\W_F)],\oplus)$. 
We set \[\G_0(\Rep_{\WD,\ss}(\W_F))=\langle [0,k-1]\otimes_{\ss}\St_0(\Z_\Psi)-[0,k-1]\otimes_{\ss}\Cyc(\Z_\Psi)\rangle_{\Z_\Psi\in \li(W_F),\ k\in \N_{\geq 1}},\] 
the additive subgroup of~$\G(\Rep_{\WD,\ss}(\W_F))$ generated by the differences $[0,k-1]\otimes_{\ss}\St_0(\Z_\Psi)-[0,k-1]\otimes_{\ss}\Cyc(\Z_\Psi)$ for $\Z_\Psi\in \li(\W_F)$ and $k\in \N_{\geq 1}$.
 
\begin{prop}\label{proposition iso 1}
The canonical map~$h_{\C}:\C_{\WD,\ss}(\W_F)\rightarrow \G(\Rep_{\WD,\ss}(\W_F))/\G_0(\Rep_{\WD,\ss}(\W_F))$, obtained by composing the canonical projection~$h:\G([\Rep_{\WD,\ss}(\W_F)])\rightarrow \G(\Rep_{\WD,\ss}(\W_F))/\G_0(\Rep_{\WD,\ss}(\W_F))$ with the 
natural injection of $\C_{\WD,\ss}(\W_F)\hookrightarrow\G(\Rep_{\WD,\ss}(\W_F))$, is injective.  Moreover, its image is stable under the operation $\oplus$. In particular, this endows the set $\C_{\WD,\ss}(\W_F)$ with 
a natural monoid structure.
\end{prop}
\begin{proof}
Note that $h_{\C}$ is the restriction of the canonical surjection $h$ to $\C_{\WD,\ss}(\W_F)$.  Let~$\Phi,\Phi'$ be~$\C$-parameters, as in Section \ref{subsection CV} and the last proof, we write 
\begin{align*}
\Phi&=\bigoplus_{k\geq 1, \Z_\Psi\in \li(\W_F)} [0,k-1]\otimes_{\ss} 
\left( (\oplus_{i=0}^{o(\Z_\Psi)-1} m_{\Z_\Psi,k,i} \nu^i \Psi)
\oplus 
 n_{\Z_\Psi,k} \Cyc(\Z_{\Psi}) \right) \\  
\Phi'&=\bigoplus_{k\geq 1, \Z_\Psi\in \li(\W_F)} [0,k-1]\otimes_{\ss} 
\left( (\oplus_{i=0}^{o(\Z_\Psi)-1} m'_{\Z_\Psi,k,i} \nu^i \Psi)
\oplus
 n'_{\Z_\Psi,k} \Cyc(\Z_{\Psi}) \right) \end{align*} where for each $(\Z_{\Psi},k)$, there are $i,i'$ such that $m_{\Z_\Psi,k,i}=0$ and $m'_{\Z_\Psi,k,i'}=0$. 
Suppose that both $\Phi$ and $\Phi'$ have same the image under $h_{\C}$, then $\Phi'-\Phi\in \Ker(h)= \G_0(\Rep_{\WD,\ss}(\W_F))$. We thus get an equality of the form 
\[\Phi- \Phi'= \bigoplus_{k\geq 1, \Z_\Psi\in \li(\W_F)} a_{\Z_\Psi,k} ([0,k-1]\otimes_{\ss}\St_0(\Z_\Psi)-[0,k-1]\otimes_{\ss}\Cyc(\Z_\Psi)),\] where all sums are finite. Set $J^+$ to be the set of 
pairs $(\Z_\Psi,k)$ such that 
$a_{\Z_\Psi,k}\geq 0$ and  $J^-$ to be the set of pairs $(\Z_\Psi,k)$ such that 
$b_{\Z_\Psi,k}:=-a_{\Z_\Psi,k}>0$. We obtain 
\begin{align*}\Phi \oplus  \bigoplus_{(\Z_\Psi,k)\in J^-}& b_{\Z_\Psi,k}[0,k-1]\otimes_{\ss}\St_0(\Z_\Psi)\oplus \bigoplus_{(\Z_\Psi,k)\in J^+} a_{\Z_\psi,k}[0,k-1]\otimes_{\ss}\Cyc(\Z_\Psi) \\
&=\Phi'\oplus \bigoplus_{(\Z_\Psi,k)\in J^-} b_{\Z_\Psi,k}[0,k-1]\otimes_{\ss}\Cyc(\Z_\Psi)\oplus \bigoplus_{(\Z_\Psi,k)\in J^+} a_{\Z_\Psi,k}[0,k-1]\otimes_{\ss}\St_0(\Z_\Psi)\end{align*} in $[\Rep_{\WD,\ss}(\W_F)]$. Now take $(\Z_\Psi,k)\in J^+$, there is $i$ 
such that $m_{\Z_\Psi,k,i}=0$. Comparing the occurence of $[0,k-1]\otimes_{\ss}\nu^{i}\Psi$ on the left and right hand sides of the equality we obtain 
\[0=m'_{\Z_\Psi,k,i}+a_{\Z_\Psi,k}\Rightarrow a_{\Z_\Psi,k}=0.\]
Hence we just proved thet $a_{\Z_\Psi,k}=0$ for all $(\Z_\Psi,k)\in J^+$. The symmetric argument shows that for $(\Z_\Psi,k)\in J^-$, there is $i'$ such that \[m_{\Z_\psi,k,i'}+b_{\Z_\Psi,k}=0\Rightarrow b_{\Z_\Psi,k}=0,\] which is impossible by assumption.  Hence $J=J^+$ and $a_{\Z_{\Psi},k}=0$ for all $\Z_\Psi\in J$, which implies $\Phi=\Phi'$, so $h_{\C}$ is indeed injective. 

For the next assertion, suppose that~$h_{\C}(\oplus_{\Phi\in [\Indec_{\WD,\ss}(\W_F)]} n_{\Phi} \Phi)\in\Image(h_{\C})$. Take $\Phi_0\in [\Indec_{\WD,\ss}(\W_F)]$ and consider 
$h_{\C}(\oplus_{\Phi\in [\Indec_{\WD,\ss}(\W_F)]} n_{\Phi} \Phi)\oplus h_{\C}(\Phi_0)$. If $\Phi_0$ ``completes a cycle'' of 
$\oplus_{\Phi\in \Indec_{\WD,\ss}(\W_F)} n_{\Phi} \Phi$, i.e. if $\Phi_0=[0,k]\otimes_{\ss} \Psi$ with $\Psi$ an irreducible representation $\Psi$ of $\W_F$, and 
if all other elements of $[0,k]\otimes_{\ss} \Z_\Psi$ appear in $\oplus_{\Phi\in \Indec_{\WD,\ss}(\W_F)} n_{\Phi} \Phi$ as 
representations $[0,k]\otimes_{\ss} \nu^j \Psi$ with corresponding multiplicities $n_{[0,k]\otimes_{\ss} \nu^j \Psi}\geq 1$, then 
setting $I=\{[0,k]\otimes_{\ss} \nu^j \Psi, j= 1,\dots, o(\Psi)-1\},$ one gets 
\[h_{\C}(\oplus_{\Phi\in [\Indec_{\WD,\ss}(\W_F)]} n_{\Phi} \Phi)\oplus h_{\C}(\Phi_0)=
h_{\C}(\oplus_{\Phi\notin I} n_{\Phi} \Phi\oplus \oplus_{\Phi\in I} (n_{\Phi}-1) \Phi \oplus \Cyc(\Z_\Psi)).\]

If $\Phi_0$ does not complete a cycle, one has 
\[h_{\C}(\oplus_{\Phi\in [\Indec_{\WD,\ss}(\W_F)]} n_{\Phi} \Phi)\oplus h_{\C}(\Phi_0)=h_{\C}(\oplus_{\Phi\in \Indec_{\WD,\ss}(\W_F)} n_{\Phi} \Phi\oplus \Phi_0).\] The assertion follows by induction.
\end{proof}

In fact the tensor product operation descends on $\Image(h_{\C})$.

\begin{prop}\label{proposition iso 2}
The additive subgroup $\G_0(\Rep_{\WD,\ss}(\W_F))$ of the ring $\G(\Rep_{\WD,\ss}(\W_F))$ is in fact an ideal. Moreover $\Image(h_{\C})$ 
is stable under $\otimes_{\ss}$. In particular this endows $\C_{\WD,\ss}(\W_F)$ with 
a natural semiring structure, and $h_{\C}$ becomes a semiring isomorphism from $\C_{\WD,\ss}(\W_F)$ to 
$\Image(h_{\C})$.
\end{prop}
\begin{proof}
For the first part, taking $\Psi_0\in \Irr(\W_F)$, it is enough to prove that for any $\Phi_1\in \Irr_{\WD,\ss}(\W_F)$ and $k,\ l\geq 0$, the tensor product 
$[0,k]\otimes_{\ss}(\St_0(\Z_{\Psi_0})-\Cyc(\Z_{\Psi_0}))\otimes_{\ss}[0,l]\otimes_{\ss}\Phi_1$ belongs to $\G_0(\Rep_{\WD,\ss}(\W_F))$. By associativity and commutativity of tensor product, and because 
$[0,i]\otimes_{\ss}[0,j]$ is always a sum of segments by Lemma \ref{lemma tensor product decomposition}, it is enough to check that $(\St_0(\Z_{\Psi_0})-\Cyc(\Z_{\Psi_0}))\otimes_{\ss}\Phi_1$ belongs to $\G_0(\Rep_{\WD,\ss}(\W_F))$. Suppose first that $\Phi_1$ is nilpotent, i.e. 
$\Phi_1=\Psi_1\in \Irr(\W_F)$. 
Because $\St_0(\Z_{\Psi_0})\otimes_{\ss}\Psi_1$ is fixed by $\nu$ under twisting and because its Deligne operator is zero, we get that 
\[\St_0(\Z_{\Psi_0})\otimes_{\ss}\Psi_1=\bigoplus_{\Z_\Psi\in \li(\W_F)} a_{\Z_\Psi}\St_0(\Z_\Psi).\] On the other hand because 
$\Cyc(\Z_{\Psi_0})\otimes_{\ss}\Psi_1$ is fixed by $\nu$ and because its Deligne operator is bijective we obtain 
\[\Cyc(\Z_{\Psi_0})\otimes_{\ss}\Psi_1=\bigoplus_{\Z_\Psi\in \li(\W_F)} b_{\Z_\Psi}\Cyc(\Z_\Psi).\] Now observing 
that both $\St_0(\Z_{\Psi_0})\otimes_{\ss}\Psi_1$ and $\Cyc(\Z_{\Psi_0})\otimes_{\ss}\Psi_1$ have the same $\Irr(\W_F)$-support, 
it implies that $a_{\Z_\Psi}=b_{\Z_\Psi}$ for all lines $\Z_\Psi$, form which we deduce that 
$(\St_0(\Z_{\Psi_0})-\Cyc(\Z_{\Psi_0}))\otimes_{\ss}\Phi_1\in \G_0(\Rep_{\WD,\ss}(\W_F))$. With the same arguments we obtain that 
$(\St_0(\Z_{\Psi_0})-\Cyc(\Z_{\Psi_0}))\otimes_{\ss}\Phi_1=0\in \G_0(\Rep_{\WD,\ss}(\W_F))$ when $\Phi_1$ is of the form 
$\Cyc(\Z_{\Psi_1})$ (because in this case both $\St_0(\Z_{\Psi_0})\otimes_{\ss}\Phi_1$ and 
$\Cyc(\Z_{\Psi_0})\otimes_{\ss}\Phi_1$ have bijective Deligne operators).
\end{proof}

The following proposition is proved in a similar, but simpler manner than the propositions above.

\begin{prop}\label{proposition isomorphism of semirings}
Let $h_{\Nilp}$ be the restriction of \[h:\G(\Rep_{\WD,\ss}(\W_F))\rightarrow \G(\Rep_{\WD,\ss}(\W_F))/\G_0(\Rep_{\WD,\ss}(\W_F))\] to $\Nilp_{\WD,\ss}(\W_F)$, then $h_{\Nilp}$ is a semiring isomorphism and $\Image(h_{\Nilp})=\Image(h_{\C})$.
\end{prop}

The above propositions have the following immediate corollary.

\begin{corollary}\label{corollary semiring isomorphism}
One has $\CV=h_{\C}^{-1}\circ h_{\Nilp}$, in particular it is a semiring isomorphism from 
$\Nilp_{\WD,\ss}(\W_F)$ to $\C_{\WD,\ss}(\W_F)$.
\end{corollary}

\subsection*{Acknowledgements} We thank the Anglo-Franco-German Network in Representation Theory and its Applications: EPSRC Grant EP/R009279/1, the GDRI ``Representation Theory'' 2016-2020, and the LMS (Research in Pairs) for financial support.

\vspace{-0,1cm}
\bibliographystyle{plain}
\bibliography{Modlfactors}

\end{document}